\documentclass[12pt]{article}

\usepackage{pb-diagram}
\usepackage[cmtip,arrow]{xy}
\usepackage{lamsarrow}
\usepackage{pb-xy}
\usepackage{authblk}
\usepackage{amsmath}
\usepackage{amsfonts}
\usepackage{amssymb}
\usepackage{amsthm}
\usepackage{authblk}
\usepackage{color}
\usepackage{textcomp}
\usepackage{lmodern}
\usepackage{mathrsfs}
\usepackage[top=2cm, bottom=2cm, left=2cm, right=2cm]{geometry}

\newcommand {\mb}{\mathbb}
\newcommand {\Z}{\mb Z}

\newcommand {\colim}{\textrm{colim}\ }

\newcommand {\lra}{\longrightarrow}

\newtheorem{thm}{Theorem}[section]

\newtheorem{cor}[thm]{Corollary}

\newtheorem{defn}[thm]{Definition}

\newtheorem{prop}[thm]{Proposition}

\newtheorem{lmm}[thm]{Lemma}

\newtheorem{exm}[thm]{Example}

\newtheorem{rem}[thm]{Remark}

\def\co{\colon\thinspace}
\def\co{\colon\thinspace}

\title{Splitting {Madsen-Tillmann
spectra II. The Steinberg idempotents and Whitehead conjecture}}

\author{Takuji Kashiwabara
\thanks{The first author was supported in part by grant ANR-08-BLAN-0248} \\
Institut Fourier, CNRS UMR \textup{5582}, Universit\'e de Grenoble I,\\
38402 St Martin d'H\`eres cedex France\\
\textit{takuji.kashiwabara@ujf-grenoble.fr}
\and
Hadi Zare
\thanks{The second author has been supported in part by IPM Grant No. 93550117. He also acknowledges partial support from the University of Tehran.}
\\
School of Mathematics, Statistics, and Computer Science, College of Science,\\
University of Tehran, Tehran, Iran \textup{14174}\\
School of Mathematics, Institute for Research in Fundamental Sciences (IPM), P.O. Box: 19395-5746, Niyavaran, Tehran, Iran\\
\textit{hadi.zare@ut.ac.ir}
}

\date{}

\begin{document}
\maketitle

\abstract{We show that, at the prime $p=2$, the spectrum $\Sigma^{-n}D(n)$ splits off the Madsen-Tillmann spectrum $MTO(n)=BO(n)^{-\gamma_n}$ which is compatible with the classic splitting of $M(n)$ off $BO(n)_+$. For $n=2$, together with our previous splitting result on Madsen-Tillmann spectra, this shows that $MTO(2)$ is homotopy equivalent to $BSO(3)_+\vee\Sigma^{-2}D(2)$.}

\tableofcontents
\section{Introduction}
The Madsen-Tillmann spectrum $MTO(n)$ is defined to be the Thom spectrum of the virtual bundle $-\gamma_n\lra BO(n)$ where $\gamma_n$ is the universal $n$-plane bundle. It is known that these spectra
filter the spectrum $MO$, i.\ e.\ there is a sequence

\begin{equation}\label{cofibMB}
S^0=MTO(0)\rightarrow \Sigma MTO(1)\rightarrow{\cdots\to} \Sigma ^{n-1}MTO(n-1)\stackrel{\iota _n}{\rightarrow}
{\Sigma^n}MTO(n)\rightarrow \cdots
\end{equation}
where $\iota _n$ is induced by the inclusion $O(n-1)\subset O(n)$, with the property $\colim \Sigma ^nMTO(n)\cong MO$.  Furthermore the cofiber of the successive stages is homotoopy equivalent to $BO(n)_+$, that is, we have a cofiber sequequnce \cite{GMTW}
\begin{equation}
\cdots \lra \Sigma^{-1} MTO(n-1)\lra MTO(n)\stackrel{p_n}{\lra }BO(n)_+\lra MTO(n-1)\lra \cdots \label{cofibgtmw}
\end{equation}
where $p_n$ is the map induced by the ``embedding'' of $(-1)$ times the canonical bundle into the $0$-dimensional trivial bundle.
In other words, the spectrum $MO$ can be built up from pieces $BO(n)_+$.

We have shown that localised away from $2$, $MTO(2n)\simeq BO(2n)_+$ and $MTO(2n+1)\simeq *$ for all $n\geqslant 0$
\cite[Theorem 1.1.{B}]{KZ1}. Thus the study of $MTO(n)$'s become essentially 2-local problems. Therefore we shall work at the
prime $p=2$. For technical reasons, we will rather work with $2$-completed spectra instead of $2$-local spectra, and in
the rest of the paper we identify a spectrum with its $2$-completion.  Since our main application concerns the mod $2$ cohomology
of associated infinite loop spaces, by passing to $2$-completion, no information will be lost.  Throughout the paper
homology and cohomology are takn with $\Z /2$ coefficients.  As we work essentially with spectra, we identify spaces with its suspension
spectra.  In the litterature, sometimes a space $X$ is identified with the suspension spectra of the space with added basepoint
$X_+$, which explains a notational discrepancy the reader may find between the current paper and results we quote.

At the prime 2,
Randal-Williams  computed
$H_*(\Omega ^{\infty} MTO(i))$ for $i=1$ and $2$, \cite[Theorems A and B]{Ra}. Combining the two theorems, we get an exact sequence
of Hopf algebras
\begin{equation}\label{exactRa}
 H_*(Q_0BO(2)_+)\rightarrow H_*(Q_0BO(1)_+)\rightarrow H_*(Q_0BO(0)_+)\rightarrow \Z /2
\end{equation}
where the (Hopf) kernel of the first two maps are isomorphic to $H_*(\Omega ^{\infty} MTO(i))$ for $i=2$ and $1$ respectively.
Thus a natural question to ask was whether if this exact sequence could be extended further to the left with
$H_*(\Omega ^{\infty} MTO(i))$ isomorphic to the kernel of each stage.  We showed that this was impossible in \cite{KZ1}.  So
a new question to ask, then, is to what extent we can generalize \cite[Theorems A and B]{Ra}.

This question leads to a search for another sequence of spectra with
the begining as in (\ref{cofibMB}).  It turns out that there indeed is such a sequence, well-known to stable homotopy theorists, that is:

\begin{equation}\label{cofibD}
S^0=D(0)\rightarrow D(1)\rightarrow{\cdots\to}D(n-1)\stackrel{\iota _n}{\rightarrow}
D(n)\rightarrow \cdots .
\end{equation}  These spectra realizes the length filtration of the Steenrod algebra, that is, we have
$$H^*(D(n))\cong \mathcal{A}/G_n\mbox{ where $G_n$ is the span of $Sq^I$, $I$ admissible, $l(I)> n$ }.$$ We note
that $G_n$ happens to be a left $\mathcal{A}$-ideal, so that this isomorphism is as $\mathcal{A}$ modules.  They originally were defined using
the symmetric powers (\cite[Proposition 4.3]{MP}), and we have $\colim D(n)=H\Z/2$, and cofibrations
\begin{equation}\label{cofibDM}
\rightarrow \Sigma ^{-n-1}M(n)\rightarrow D(n-1)\rightarrow D(n)\rightarrow \Sigma ^{-n}M(n)\rightarrow
\end{equation}

%
The spectrum $M(n)$ is defined to be the cofibre of the map $\Sigma^{-n}D(n-1)\lra \Sigma^{-n}D(n)$.
Thus one can say that $H\Z/2$ can be built up of $M(n)$'s. These building blocks can be described alternatively as follows.

The spectra $BO(1)^{\times n}_+$ admit a natural (left) $Gl_n(\Z/2)$ action.  Thus the Steinberg idempotent
$e_n\in \Z/2[Gl_n(\Z/2)]$ gives rise to a splitting of $BO(1)^{\times n}_+$ and we have $M(n)\simeq e_n BO(1)^{\times n}_+$
 \cite[Theorem 5.1]{MP}.
 Moreover, through the Becker-Goettlieb transfer map, this splitting gives rise to a splitting of
$M(n)$ off
 $BO(n)_+$ (see Mitchell and Priddy's paper \cite{MP} for more details and for the odd primary splitting results as well).

Therefore we have a construction of $MO$ with $BO(n)_+$'s as building blocks, a construction of $H\Z/2$ with
$M(n)$'s as building blocks.  Furhtermore $H\Z/2$ and $M(n)$'s split off respectively $MO$ and $BO(n)_+$'s.  It is then natural
to ask whether one can split intermediate stages.  The purpose of this paper is to answer affirmatively to this question, and
discuss some consequences, including an answer to the question on generalization of the exact sequence (\ref{exactRa}).  More detailed
statements are given in the next section.

The paper is organized as follows.  In section \ref{sec:res}, we summarize our results.
  In section \ref{sec:whi}, we recall relevant results
from \cite{KP1} and construct a map from $F_*Y$ to $F_*X$.  In section \ref{sec:tak} we use Takayasu's results \cite{Takayasu} to
construct a map going the other way round, and show that we indeed have a splitting.  In section \ref{sec:hom} we discuss the
consequences in homology of infinite loop spaces.

The first author  thanks  Andrew Baker, Masaki Kameko, Nick Kuhn, Bob Oliver, Stewart Priddy, Oscar Randal-Williams and Lionel
Schwartz for helpful conversations. {The second author is grateful to Institut Fourier for its hospitality and support for a visit during October 2014.}

\section{Statement of results}\label{sec:res}
To give precise statements, we start with some defintions.
\begin{defn}
\begin{enumerate}
 \item A filtered spectrum is a sequence of spectra $ F_*X$
\begin{equation}
 F_0X\rightarrow F_1X \rightarrow \cdots \rightarrow F_nX \rightarrow F_{n+1}X\rightarrow \cdots
\end{equation}
with a homotopy equivalence $hocolim F_nX\simeq X$.
\item A map of filtered spectra $(f_*)$ from $ F_*X$ to $ F_*Y$ is a collection of maps $f_n\co F_nX\to F_nY$
that makes the squares
$$\begin{diagram}
   \node{F_nX}\arrow{e}\arrow{s}  \node{F_{n+1}X}\arrow{s}\\
\node{F_nY}\arrow{e} \node{F_{n+1}Y}
  \end{diagram}$$ commutative.
\item Two filtered spectrum $ F_*X$ and $ F_*Y$ are said to be equivalent if there is a map $f_*$ from $ F_*X$ to $ F_*Y$
 such that $f_n$ 's are homotopy equivalence for all $n$.
\item We say that $F_*X $ splits off $F_*Y$ if there are maps $f_*$ from $ F_*X$ to $ F_*Y$
and $g_*$   from $ F_*Y$ to $ F_*X$ such that $H_*(f_n\circ g_n)=id$.

\end{enumerate}

\end{defn}

Our main result then reads as follows.
\begin{thm}\label{split}
Define filtered spectra $F_*X$ and $F_*Y$ by  $F_nX=D(n)$, $F_nY=\Sigma ^nMTO(n)$.  Then $F_*X$ splits off $F_*Y$.

\end{thm}


As our proof doesn't depend on the decomposition of $MO$ (\cite[Theorem 2.10]{Th},\cite[Theorem 2]{Ta}), we obtain a ``new''
proof of the splitting of $H\Z/2$ off $MTO$.  However, our method doesn't allow us to obtain information on other summands.

An immediate consequence is the following.
\begin{cor}
$H_*(\Omega ^{\infty} \Sigma ^{-n}D(n))$ splits off $H_*(\Omega ^{\infty} MTO(n))$  as a Hopf algebra.
\end{cor}

Thus the ``correct way to extend'' the exact sequence (\ref{exactRa}) is just the following standard fact

\begin{prop}[\cite{KP1}]\label{dnhomologyexact}
The following sequence of Hopf algebras is exact.
$$\cdots \rightarrow H_*(\Omega ^{\infty}M(n))\rightarrow  H_*(\Omega ^{\infty}M(n-1))\rightarrow \cdots \rightarrow
H_*(\Omega ^{\infty}M(2))\rightarrow H_*(Q_0B\Z/2_+) \rightarrow H_*(Q_0S^0)\rightarrow \Z /2$$
Furthermore the image of $H_*(\Omega ^{\infty}M(n))\rightarrow H_*(\Omega ^{\infty}M(n-1))$ is isomorphic to
$H_*(\Omega ^{\infty}_0D(n-1))$.
\end{prop}

As $D(0)\cong S^0$, $\Sigma ^{-1}D(1)\cong MTO(1)$ \cite[Proposition 4.4]{MP}, and $M(1)\cong BO(1)_+$,
combined with the $n=2$ case of Theorem \ref{split}, we recover Theorems A and B of \cite{Ra}.

Of course, the cohomology being dual of homology, the exact sequences above give some information on certain characteristic classes.
More precisely, recall from \cite{Ra,KZ1}
\begin{defn}
A universally defined characteristic class  in $H^*(\Omega ^{\infty }_0MTO(n))$ is a element in the subalgebra generated by the image of $H^*(BO(n))\rightarrow H^*(Q_0BO(n))\rightarrow H^*(\Omega ^{\infty }_0MTO(n))$.  We denote $\mu _{i_1,\ldots ,i_n}=(\Omega ^{\infty}p_n)^*(\sigma ^{\infty *}(\sigma _1^{i_1}\ldots \sigma _n^{i_n}))$ where $H^*(BO(n))\cong \Z/2[\sigma _1,\ldots ,\sigma _n]$.
 \end{defn}

In \cite{KZ1}, we used the summand $BSO(2n+1)_+$ that split off $MTO(2n)$ to show that some of these classes remain algebraically
independent.  Here we use the summand $D(n)$ that splits off $MTO(n)$ to show that there are ``linear'' relations corresponding to
elements of $H^*(M(n))$, and that in the case of dimension 2, these relations together with the ones derived from the action of
top Steerond squares are the only relations.  More precisely, we will show:
\begin{thm}\label{relfromdn}
\begin{enumerate}
 \item In $H^*(\Omega ^{\infty }_0MTO(n))$, we have relations $(\Omega ^{\infty}p_n)^*(\sigma ^{\infty *}(x))=0$ for
$x\in H^*(M(n))\subset H^*(BO(n))$.
\item For $n=2$, the only relations among $\mu _{i,j}$'s are the relations above, and $\mu _{2i,2j}=\mu _{i,j}^2$.
\item Again for $n=2$, the subalgebra of universally defined characteristic classes in $H^*(\Omega ^{\infty }_0MTO(2))$
is the polynomial algebra generated by $\nu _{i,j}$'s with $ij$ odd, where $\nu _{i,j}$ is defined in \cite{KZ1}.
\end{enumerate}
\end{thm}
We will give more precise description of the inclusion $H^*(M(n))\subset H^*(BO(n))$ later.


\section{Maps from $MTO(n)$ to $\Sigma ^{-n}D(n)$}\label{sec:whi}
In this section we use results from \cite{KP1} to construct maps from $\Sigma ^nMTO(n)$'s to $D(n)$'s that form
a map of filtered spectra.  First of all, we recall reults we will need.
\subsection{Exact sequences of spectra and the Whitehad conjecture}
We start with a definition.
\begin{defn}[\cite{KP1}]
\begin{enumerate}
 \item A fibration sequence of spectra $F\rightarrow X\stackrel{f}{\rightarrow} Y$ is called exact if there exists a map
$g:\Omega ^{\infty}Y\lra \Omega ^{\infty}X$ such that $\Omega ^{\infty}f\circ g\simeq id$.
\item A sequence of spectra $\rightarrow \cdots \rightarrow X_n\rightarrow \cdots X_1\rightarrow X_0\rightarrow E_{-1}$
is called exact if for all $n$, $E_n\rightarrow X_n \rightarrow E_{n-1}$ is exact, where $E_n$ is inductively defined as
the fiber of the map
$X_n \rightarrow E_{n-1}$.
\end{enumerate}
\end{defn}
The category of spectra being a triangulated category instead of an abelian category, we have some complication here.  The
notion of exactness with three terms is more or less a counterpart of a split short exactness in abelian categories.  The use of this
seemingly too strong condition is motivated by the following fact.
By definition, an exact sequence of spectra yields an exact sequence of abelian groups
 upon applying $[Y,-]$ for a suspension spectrum $Y$, or a spectrum that is a summand of a suspension spectrum.
Thus one can regard suspension spectra as free objects, summands of suspension spectra as projective objects,
and carry out homological algebra in the category of spectra.

Now, one of the main results of \cite{KP1}, reads as follows.

\begin{thm}[{\cite[Theorem 1.1]{KP1}}]\label{whitehead}
Let $d_k$ be defined by the composition
$$d_k:M(k+1)\hookrightarrow BT(k+1)_+=B(T(1)\times T(k)) _+ \stackrel{tr}{\rightarrow }BT(k)_+\rightarrow M(k)$$
where $tr$ is the Becker-Gottlieb transfer, and the first and last map are obtained via the
splitting  $M(k) \simeq e_k'BT(k)_+$ with the conjugate Steinberg idempotent $e_k'$. Then the sequence
\begin{equation}\label{seq:wh}
 \cdots \stackrel{d_{k+1}}{\rightarrow} M(k+1)\stackrel{d_{k}}{\rightarrow} M(k)\rightarrow \cdots \rightarrow M(1)\stackrel{d_0}{\rightarrow} M(0)\stackrel{\epsilon}{\rightarrow} H\Z/2
\end{equation}
is exact.
\end{thm}
We note that in \cite{KP1}, the sequence above is shown to be equivalent to another sequence consisting of $M(k+1)$'s,
whose exactness is known as the mod $2$ Whitehead conjecture (Corollary 1.2 {\it loc.cit.}).

\subsection{{Complexes of spectra}}
We still need some more definitions.
\begin{defn}
\begin{enumerate}
 \item By a chain complex of
 spectra $(E_n,d_n)$ we understand a sequence of spectra $E_n$ with maps $d_{n-1}:E_n\to E_{n-1}$ so that the
 composition $E_{n+1}\to E_n\to E_{n-1}$ is null for all $n$. By a map $f$ of chain complexes of spectra $(E_n,d_E)\to
 (F_n,d_F)$ we mean a collection of maps $f_n:E_n\to F_n$ such that $f_nd_E=d_Ff_{n+1}$. Two complexes are
said to be isomorphic if if there are maps $f$ from $E$ to $F$  such that all $f_n$'s are homotopy equivalences.
\item  Let $F_*X$ be a filtered spectrum.    Define its {\it associated graded} complex $Gr_{\bullet}(F_*X)$ by
$Gr_0(F_*X)=F_0X$, $Gr_i(F_*X)=\Sigma ^{-i} cofib(F_{i-1}X\to F_iX)$, with obvious maps $Gr_iF_*X\to Gr_{i-1}F_*X$.
\end{enumerate}
\end{defn}
\begin{rem}
It follows from an easy diagram chasing that two filtered spectra are equivalent if and only if their associated graded complexes
are isomorphic.  Thus $Gr$ provides an embedding of the category of filtered spectra into that of complexes of spectra.
\end{rem}
\begin{exm}

\begin{enumerate}
 \item Let $F_nX=D(n)$.  Then the associated graded complex $Gr_{\bullet}(F_*X)$ is
$$\cdots \to M(n+1)\stackrel{\delta _n}{\to} M(n)\to \cdots \to M(0)$$ considered in \cite[Corollary 1.2]{KP1}.  This complex,
with augmentation to $H\Z/2$ added,
was shown to be equivalent with the complex \ref{seq:wh} in \cite[section 6]{KP1}.
\item Let $F_nY=\Sigma ^nMTO(n)$.   Then the associated graded complex $Gr_{\bullet}(F_*X)$ is given by
$(BO(n)_+,tr)$ where
$tr$ is the Becker-Gottlieb transfer associated to the inclusion
$O(n-1)\subset O(n)$, as the Becker-Gottlieb transfer $BO(n)\to BO(n-1)$ factors as $BO(n)_+\to MTO(n-1)\to BO(n-1)_+$.
We can also see that $(BO(n)_+,tr)$ is a complex directly by noting that as $O(n)\subset O(n)\times O(2) \subset
O(n+2)$, thus $S^1\cong \{1\}\times SO(2)\subset O(n)\times O(2)$ normalizes $O(n)\subset O(n+2)$,
so the transfer associated to $O(n)\subset O(n+2)$ is trivial.  Moreover, the complex of free spectra
 $(BO(n)_+,tr)$ is augmented over $H\Z/2$ since the composition $BO(1)_+\rightarrow BO(0)_+\rightarrow H\Z/2$ is
 trivial. This is just another way of saying that the transfer in $\Z/2$-cohomology $H^*(BO(0)_+;\Z/2)\rightarrow
 H^*(BO(1)_+;\Z/2)$ is trivial.
\end{enumerate}

\end{exm}

{Now we are ready to do homological algebra.
We recall the following.}

\begin{prop}\label{analogy}
Let $(P_{\bullet },d_{\bullet })$ be a cochain complex of projective $R$-modules with an augmentation
$P_0\rightarrow A$, and $(A_{\bullet },d_{\bullet })$ be a resolution of $A$.  Then we get a cochain map
from $(P_{\bullet },d_{\bullet })$ to $(A_{\bullet },d_{\bullet })$ .
\end{prop}
This is an easy exercise using the definition of projectives and exactness, and proof is omitted.
We now translate this to our setting.

\begin{prop}\label{analogy1}
Let $(P_{\bullet },d_{\bullet })$ be a chain complex of projective spectra with an augmentation
$P_0\rightarrow A$, and $(A_{\bullet },d_{\bullet })$ be a resolution of $A$.  Then we get a map of chain complexes
from $(P_{\bullet },d_{\bullet })$ to $(A_{\bullet },d_{\bullet })$ .
\end{prop}
\begin{proof}
First, note that Proposition \ref{analogy} is proved using the fact that if $P$ is projective and $C_0\rightarrow C_1\rightarrow C_2$
is a short exact sequence, then $Hom(P,C_0)\rightarrow Hom(P,C_1)\rightarrow Hom(P,C_2)$ is short exact. Thus it suffices to show that
if $P$ is a suspension spectrum (or a summand of a suspension spectrum, but this doesn't affect anything), and $C_0\rightarrow C_1\rightarrow C_2$
is a short exact sequence of spectra, then $[P,C_0]\rightarrow [P,C_1]\rightarrow [P,C_2]$ is a short exact sequence
of abelian groups.  Write $P=\Sigma ^{\infty}X$.  Then we have $[P,C_i]\cong [X,\Omega ^{\infty}C_i]$.
By definition of the short exactness, the map $\Omega ^{\infty}C_1\rightarrow \Omega ^{\infty}C_2$ (as well as loop on this)
admits a section,
thus $[P,C_1]\cong [X,\Omega ^{\infty}C_1]\rightarrow  [X,\Omega ^{\infty}C_2]\cong
[P,C_2]$ is (split-)epi.  Since $C_0\rightarrow C_1\rightarrow C_2$ is a cofibration of
spectra, it follows that $[P,C_0]\rightarrow [P,C_1]\rightarrow [P,C_2]$ is a short exact.
\end{proof}

\subsection{{Construction of the maps}}
{By the examples of previous section, and Proposition \ref{analogy1}, we have a maps of chain complexes of spectra $(BO(n)_+,tr)\to (M(n),d_n)$
That is,} 
we have proven the existence of maps $f_n$ such that the following square commutes

$$\begin{diagram}
\node{BO(n)_+}\arrow{s,r}{f_n}\arrow{e,t}{tr}\node{BO(n-1)_+}\arrow{s,r}{f_{n-1}}\\
\node{M(n)}\arrow{e,t}{d_{n-1}}\node{M(n-1)}
  \end{diagram}
$$
\begin{rem}
The spectrum $M(0)$ is just $S^0$, and $M(1)=BO(1)_+$. The maps $BO(0)_+\to H\Z/2$ and $\epsilon:M(0)\to H\Z/2$ coincide with the unit of $H\Z/2$, and the maps $f_1$ and $f_0$ can be taken to be the identity. 
\end{rem}
Now we are ready to prove the following.

\begin{thm}\label{compatibility1}
{Fix maps $f_n:(BO(n)_+,tr )\to (M(n),d_n)$. Then t}here exists a map \\ $\alpha _n \co MTO(n)
\rightarrow \Sigma ^{-n}D(n)$ which makes the following diagram commutative for each $n$
$$
\begin{diagram}
 \node{MTO(n)}\arrow{s,t}{\alpha _n}\arrow{e}   \node{BO(n)_+}\arrow{s,t}{{f_n}}\\
\node{\Sigma ^{-n}D(n)}\arrow{e}                \node{M(n).}
\end{diagram}
$$
\end{thm}
\begin{proof}
We proceed by induction on $n$.  The case $n=0$ is trivial.
Suppose that we have constructed such $\alpha _{n-1}$.  Consider the following diagram.
$$
\begin{diagram}
 \node{BO(n)_+}\arrow{e}\arrow{s,l}{f_n}\node{MTO(n-1)}\arrow{s,r}{{\alpha_{n-1}}}\\
 \node{M(n)}\arrow{e}\node{\Sigma ^{1-n}D(n-1)}
\end{diagram}
$$
If we can show that this diagram commutes, then we can define the map $\alpha _n$ using the cofibrations
(\ref{cofibMB}) and (\ref{cofibDM}), which will conclude the proof.
 Note that the two horizontal maps induces trivial maps in cohomology, which implies that the two compositions from the top left corner to bottom right corner factor through
$E_{n-1}$.  Here $E_n$ is the spectrum defined in \cite{KP1}, in other words, $E_n$ is the fiber of the map
$\Sigma ^nD(n)\rightarrow \Sigma ^nH\Z/2$. Thus we need to show that the two elements in $[BO(n)_+,E_{n-1}]$ agree.  However, by \cite[Theorem 1]{KP1} , $[BO(n)_+,E_{n-1}]$ injects to $[BO(n)_+,M({n-1})]$.  Thus it suffices to show that the two maps agree after composition with the map
$E_{n-1}\rightarrow \Sigma ^{1-n}D(n-1) \rightarrow M({n-1})$. Now, consider the following diagram.
$$
\begin{diagram}
 \node{BO(n)_+}\arrow{e}\arrow{s,l}{f_n}\node{MTO(n-1)}\arrow{s,l}
{\alpha _{n-1}}\arrow{e}\node{BO(n-1)_+}\arrow{s,l}{f_{n-1}} \\
 \node{M(n)}\arrow{e}\node{\Sigma ^{1-n}D(n-1)}\arrow{e}\node{M(n-1)}
\end{diagram}
$$
The right square is commutative by inductive hypothesis.  But we chose our maps $f_n$ so that the big square commutes.
This finishes the proof.

\end{proof}
Note that by construction, the maps $\alpha _n$ form a map of filtered spectra.

\begin{rem}
 A reader familiar with \cite{KP1} must have noticed that our arguments are slightly upside-down.  If one follows the proofs in
\cite{KP1}, we get $\alpha _{n-1}$ before $f_n$.  As main interests of the authors of {\it loc.cit.} were not on $D(n)$'s, some
statements that could have been proved in {\it loc.cit} and that we could have quoted are not there.  We chose to quote the
statements that could be easily found, instead of details of proofs.
\end{rem}

\section{The splitting}\label{sec:tak}
In this section, we construct a map of filtered spectra from $D(n)$'s to $\Sigma ^{n}MTO(n)$'s and conclude the proof of
Theorem \ref{split}.  The main ingredient here is the description of $D(n)$ as a summand of Thom spectrum, which is implicit
in \cite{Takayasu}
\subsection{$D(n)$ as a summand of Thom spectrum}
Consider the reduced regular representation $\rho_n$ of $T(n)=O(1)^{\times n}$ and let $M(n)_k$ be the stable summand of the Thom spectrum $BT(n)^{k\rho_n}$ corresponding to the Steinberg idempotent $e_n$. In particular, $M(n)_0=M(n)$ defined in previous section. According to Takayasu \cite{Takayasu} there is a cofibration of spectra
\begin{equation}\label{Takcofib}
\Sigma^kM(n-1)_{2k+1}\lra M(n)_k\lra M(n)_{k+1}
\end{equation}
where the mapping $M(n)_k\stackrel{j_k}{\to} M(n)_{k+1}$ is induced by the bundle map $k\rho_n\to (k+1)\rho_n$. We are interested in the case $k=-1$ where Takayasu's cofibration looks like
\begin{equation}\label{cof:tak}
 \Sigma^{-1}M(n-1)_{-1}\stackrel{i'_{-1}}\lra M(n)_{-1}\stackrel{j_{-1}}\lra M(n)_0.
\end{equation}
Furthermore, the map
$i'_{-1}$ satisfies the property $i'_{-1}\circ j_{-2}=i_{-1}$, where $ j_{-2}:\Sigma ^{-1}M(n-1)_{-2}\rightarrow M(n)_{-1}$ is induced by the
inclusion $T(n-1)\subset T(n)$ \cite[Theorem A]{Takayasu}.
We record the following observation which is implicit in \cite{Takayasu}
\begin{lmm}\label{splitting1}
There is a homotopy equivalence $M(n)_{-1}\lra\Sigma^{-n}D(n)$, i.e. $\Sigma^{-n}D(n)$ is a stable summand of $BT(n)^{-\rho_n}$ corresponding to the Steinberg idempotent.
\end{lmm}

\begin{proof}
By \cite[Proposition 4.1.6]{Takayasu}, we have $H^*(M(n)_{-1})\cong H^*(D(n))$. By Corollary 4.2.3 {\it loc.cit.},
$j_{-1}^*$ is monomorphism, so by the long exact sequence for the cofibration (\ref{cof:tak}) $i_{-1}^{\prime *}$ is epimorphism,
and the filtered spectrum
$$M(0)_{-1}\rightarrow \cdots \rightarrow \Sigma ^{n-1}M(n-1)_{-1}\rightarrow \Sigma ^nM(n)_{-1}\rightarrow \cdots $$
realizes the length filtration of the Steenrod algebra.  Thus by \cite[Corollary 1.4.1]{HK} is equivalent to
$$D(0)\rightarrow \cdots \rightarrow D(n)\rightarrow D(n+1)\rightarrow \cdots .$$
\end{proof}
\begin{rem}
 The above can also be proved by direct cohomology calculation using \cite[Proposition 4.1.6]{Takayasu} and
\cite[Theorem 5.8]{MP}.
\end{rem}

\subsection{Maps from $\Sigma^{-n}D(n)$ to $MTO(n)$}

Denote by $\beta _n$ the composition
$$\Sigma^{-n}D(n) \lra BT(n)^{-\rho_n}\lra BT(n)^{(-\gamma_1)^{\times n}}\lra BO(n)^{-\gamma_n}\cong MTO(n)$$
where  the map at the middle is induced by  the embedding $\gamma_1^{\times n}\subset \rho_n$
(or the embedding of virtual vector bundles $-\rho_n\subset (-\gamma_1)^{\times n}$ ), and
the last map is induced by the inclusion $T(n)\lra O(n)$.

Unlike the maps $\alpha _n$ that are constructed as maps between cofibers, the compatibility between different $\beta _n$'s are
not immediate from the definition.  Thus our first task is to show that they form a map of filtered spectra.

\begin{lmm}\label{compatilbe_beta}
We have following commutative diagram.  Thus the maps $\beta _n$'s form a map of filtered spectra.
$$
\begin{diagram}
 \node{D(n-1)}\arrow{e}\arrow{s,l}{\beta _{n-1}} \node{D(n)}\arrow{e}\arrow{s,l}{\beta _{n}} \node{\Sigma ^nM(n)}\arrow{s}\\
 \node{\Sigma ^{n-1}MTO(n-1)}\arrow{e} \node{\Sigma ^nMTO(n)}\arrow{e}  \node{\Sigma ^n{BO(n)}}
\end{diagram}
$$
\end{lmm}
\begin{proof}
First consider the square on the left.  Since the target of the two compositions is $(-1)$-connected, and the fiber of the map
$j_{-2}:\Sigma ^{n-1}M(n-1)_{-2}\rightarrow \Sigma ^{n-1}M(n-1)_{-1}=D(n-1)$ has no cell in positive dimension, we see that it suffices to
show that they become homotopic after the composition with $j_{-2}$.  Now, consider the diagram
$$
\begin{diagram}
 \node{D(n-1)}\arrow{e}\node{\Sigma ^{n-1}BT(n-1)^{-\rho _{n-1}}}\arrow{e}\node{\Sigma ^{n-1}BT(n-1)^{-\gamma _1^{n-1}}}\arrow{e}
 \node{\Sigma ^{n-1}BO(n-1)^{-\gamma _{n-1}}}\\
\node{\Sigma ^{n-1}M(n-1)_{-2}}\arrow{e}\arrow{n,r}{j_{-2}}\arrow{s,r}{i_{-1}}\node{\Sigma ^{n-1}BT(n-1)^{-2\rho _{n-1}}}\arrow{e}\arrow{n}\arrow{s}\node{\Sigma ^{n-1}BT(n-1)^{-\gamma _1^{n-1}}}\arrow{e}\arrow{n,=,-}\arrow{s}
 \node{\Sigma ^{n-1}BO(n-1)^{-\gamma _{n-1}}}\arrow{n,=,-}\arrow{s}\\
\node{D(n)}\arrow{e}\node{\Sigma ^nBT(n)^{-\rho _{n}}}\arrow{e}\node{\Sigma ^nBT(n)^{-\gamma _1^{n}}}\arrow{e}
 \node{\Sigma ^n BO(n)^{-\gamma _{n}}}
\end{diagram}
$$
Everything except on the left commutes by naturality of Thom spectra construction.  The top left square commutes because
the map on the right is $GL_{n-1}(\Z/2)$-equivariant.  The bottom left square commutes because of the equality
$e_{n-1}e_n=e_n$ (\cite[Corollary 2.6 (2)]{KP1}, this also follows easily from \cite[Proposition 2.5]{MP}) and the fact that the
map on the right is $GL_{n-1}(\Z/2)$-equivariant.  By \cite[Theorem A]{Takayasu}, the map $i'_{-1}:D(n-1)\rightarrow D(n)$
satisfies the property $i'_{-1}\circ j_{-2}=i_{-1}$, so the conclusion follows.
A similar but easier argument applies to the map $D(n)\rightarrow \Sigma ^n M(n)$.
\end{proof}

\subsection{Proof of the splitting}\label{homologyofmaps}
So far, we have defined the maps $\alpha_n$ and $\beta _n$ so that the following diagram commutes.

$$
\begin{diagram}\node{S^0=D(0)}\arrow{e} \arrow{s,r}{\beta_0}            \node{D(1)}\arrow{e} \arrow{s,r}{\beta_1}                \node{\cdots}\arrow{e}\node{D(n)}\arrow{e}\arrow{s,r}{\beta_n}\node{\cdots}\\
\node{S^0=MTO(0)}\arrow{s,r}{\alpha_0}\arrow{e} \node{\Sigma MTO(1)}\arrow{s,r}{\alpha_1}\arrow{e}\node{\cdots}\arrow{e}\node{\Sigma^nMTO(n)}\arrow{s,r}{\alpha_n}\arrow{e}\node{\cdots}\\
\node{S^0=D(0)}\arrow{e}             \node{D(1)}\arrow{e}                 \node{\cdots}\arrow{e}\node{D(n)}\arrow{e}\node{\cdots}
\end{diagram}
$$
Now, consider $H^*(\alpha _n\circ \beta _n)$.  As $H^*(D(n))$ is monogenic over the Steenrod algebra \cite[Proposition 4.3]{MP},
this map is determined by its image on the bottom class.  However, the bottom cell of $D(n)$ is just the image of
$D(0)$, so the bottom class is detected by the pull-back to $H^0(D(0))$.  The commutativity of the above diagram then
implies that $H^0(\alpha _n\circ \beta _n)$ restricted to $H^0(D(0))$ is identity.  Thus $H^*(\alpha _n\circ \beta _n)$ is identity,
which concludes the proof of Theorem \ref{split}

\subsection{{Further refinements}}
We have shown in \cite{KZ1} that $BSO(2n+1)_+$ splits off $MTO(2n)$. One may ask how
this splitting ineracts with the spliting of current paper.  We show that they are complementary.
\begin{cor}\label{splitbsdn}
$\Sigma^{-2n}D(2n)\vee BSO(2n+1)_+$ splits off $MTO(2n)$.  When $n=1$, we have  homotopy equivalence
$MTO(2)\cong \Sigma^{-2}D(2)\vee BSO(3)_+$.
\end{cor}
\begin{proof}
Denote by $f_{2n}$ the inclusion $O(2n)\subset SO(2n+1)$ given by $A\mapsto det(A)(A\oplus 1)$, and by
$Tr_{f_{2n}}$ the associated Miller-Mann-Mann transfer $BSO(2n+1)_+\lra MTO(2n)$,
Consider the composition
$$(\alpha _{2n}\vee Bf_{2n}\circ p_{2n})^* \circ (\beta _{2n}\vee Tr_{f_{2n}})^* :H^*(BSO(2n+1))\oplus H^*(\Sigma^{-2n}D(2n))\lra H^*(BSO(2n+1))\oplus H^*(\Sigma^{-2n}D(2n)).$$
The components $H^*(BSO(2n+1))\lra H^*(BSO(2n+1))$ and $H^*(\Sigma^{-2n}D(2n))\lra H^*(\Sigma^{-2n}D(2n))$ are automorphisms. Consider now the component
$H^*(\Sigma^{-2n}D(2n))\lra H^*(BSO(2n+1))$.  This is trivial since the source is generated over the Steenrod algebra by a
negative-degree elements, and the target is concentrated in non-negative degrees.  Thus the map
$(\alpha _{2n}\vee Bf_{2n}\circ p_{2n})^* \circ (\beta _{2n}\vee Tr_{f_{2n}})^*$ is an automorphism. This proves the splitting for general $n$.  When $n=1$,
it suffices to compare the cohomology of both sides, or alternatively, by comparing the fibrations
$MTO(2)\longrightarrow BO(2)_+\lra MTO(1)$ and $\Sigma ^{-2}D(2)\rightarrow M(2)\rightarrow D(1)$, noting that $BO(2)_+\cong M(2)\vee BSO(3)_+$ (c.f. \cite[Theorem C]{MPold}), we see that $(\alpha _{2n}\vee Bf_{2n}\circ p_{2n})^* $ induces mod $2$ homology equivalence.  Since everything in sight is of finite type, this implies that we have a $2$-local homotopy equivalence.
\end{proof}

\section{Homology of the associated infinite loop spaces}\label{sec:hom}
In this section, we discuss the consequences of our splitting theorem to the homology of associated infinite loop spaces.
\subsection{Exact sequences}
We start with generalities on summands of suspension spectra.
\begin{lmm}\label{homologysummand}
 Let $M$ be a spectrum such that $M$ splits of $\Sigma ^{\infty }X$ where $X$ is a space.  Denote $B$ a basis of
$\tilde{H}_*(M)$.  Then there are elements $s_x\in H_*(\Omega ^{\infty }M)$ such that
$$H_*(\Omega ^{\infty }M)=\Z/2[Q^I(s_x);excess(I)>|x|, I\mbox{ allowable}], \sigma ^{\infty} _*(s_x)=x,$$
where $Q^I$ denotes the Dyer-Lashof operation (\cite{Maybook}).
Suppose further that the splitting of $M$ is obtained by an idempotent of the form $f=\Sigma _if_i$, where $f_i$'s are self-maps
of the space $X$.
  Then we can choose $s_x$ in such a
way that in $H_*(QX)$ we have $s_x=x$ modulo decomposables in $H_*(QX)$, where we identify $H_*(X)$ with its image in
$H_*(QX)$ via the canonical map $X\rightarrow QX$, and $H_*(\Omega ^{\infty}X)$ with its image in $H_*(QX)$ via the splitting map.
\end{lmm}
\begin{proof}
 The first statement is straightforward, so its proof is omitted. The second statement follows as we have
$$H_*(\Omega ^{\infty}(f))=H_*(\Omega ^{\infty}(\Sigma _i f_i))\equiv \Sigma _iH_*(\Omega ^{\infty}(f_i))$$ modulo
decomposables.
\end{proof}
Thus $H_*(\Omega ^{\infty}M(n))$'s are polynomial Hopf algebra.  As a matter of fact they are bipolynomial, but we will not
need this.  A nice feature of such Hopf algebras is that a short exact sequence involving only these Hopf algebras always splits
as algebras, thus inducing a short exact sequence of indecomposables.  Besides, the map showing up in such a short exact sequence
can be effectively studied by studying the induced map in the indecomposables, with little loss of information.  All these considerations
 lead to the following refinement of Proposition \ref{dnhomologyexact}
\begin{prop}\label{exactHopf}
 The following sequence of Hopf algebra is exact.  It gives  rise to an exact sequence of graded vector spaces after taking the module
of indecomposables.
$$\cdots \rightarrow H_*(\Omega ^{\infty}M(n))\rightarrow H_*(\Omega ^{\infty}M(n-1))\rightarrow \cdots
H_*(\Omega ^{\infty}M(2))\rightarrow H_*(\Omega ^{\infty}_0BZ/2_+)\rightarrow H_*(Q_0S^0)\rightarrow \Z/2$$
Furthermore the image of $H_*(\Omega ^{\infty}M(n))\rightarrow H_*(\Omega ^{\infty}M(n-1))$ is isomorphic to
$H_*(\Omega ^{\infty}_0D(n-1))$.
\end{prop}
\begin{proof}Denote by $D'(n)$ the fiber of the map $\Sigma
 ^{-n}D(n)\rightarrow \Sigma ^{-n}H\Z/2$ corresponding to the bottom class.
Then by theorem \ref{whitehead}, we see that
the fibration $D'(n)\rightarrow M(n)\rightarrow D'(n-1)$ is exact
Note that $\Omega ^{\infty}D'(n)\cong \Omega ^{\infty}\Sigma ^{-n}D(n)$ for $n\geq 1$, and $\Omega ^{\infty}D'(0)\cong
 Q_0S^0$.  By definiition a three-term exact sequence of spectra leads to a short exact sequence of Hopf algebras
by applying $H_*(\Omega ^{\infty}(-);\Z /2)$. Splicing together the short exact sequnce thus obtained,  we get
an exact sequence as in the statement of Proposition, except  the last entries which are
$$\cdots H_*(\Omega ^{\infty}BZ/2_+)\rightarrow H_*(QS^0)\rightarrow \Z/2[\Z/2] \rightarrow \Z/2.$$
Noting that we have $ H_*(\Omega ^{\infty}X_+)\cong H_*(\Omega ^{\infty}_0X_+)\otimes \Z/2[\Z]$ for connected $X$, in
particular $X=B\Z/2$ and $X=pt$, and that
$\Z/2[\Z]\rightarrow \Z/2[\Z]\rightarrow \Z/2[\Z/2]$ is exact, we see that the sequence in the Proposition is exact.  As everything
insight is polynomial, they remain exact after passing to indecomposables.
\end{proof}
\begin{rem}
 The $Gl_n(\Z/2)$ action on $BT(n)_+$ extends that on  $BT(n)$.  Thus it is easy to see from the definition of $e'_n$ (\cite{MP})
that we have $e'_nBT(n)=e'_nBT(n)_+$ for $n>1$.  Thus $\Omega ^{\infty }_0M(n)=\Omega ^{\infty }M(n)$ for $n>1$,
and the above exact sequence can be expressed entirely in terms of $\Omega ^{\infty }_0M(n)$'s.
\end{rem}



An immediate consequence is
\begin{cor}\label{mto2poly}
$H^*(\Omega ^{\infty}_0MTO(2))$ is a polynomial algebra.
\end{cor}
\begin{proof}
By Corollary \ref{splitbsdn} we have $\Omega ^{\infty}_0MTO(2)\cong Q_0BSO(3)_+\times \Omega ^{\infty }D'(2)$, noting
that $\pi _0(D'(2))=0$ since it is a direct factor of $\pi _0(M(2))$.  The short exactsequence above implies that
$H^*(\Omega ^{\infty }D'(2))$ injects to $H^*(\Omega ^{\infty }M(3))$.  Since $M(3)$ is a stable summand of $BO(3)$, we see
that $H^*(\Omega ^{\infty }D'(2))$ injects to $H^*(Q_0BO(3))$ which is polynomial (\cite[Theorem 3.11]{W}).  Since $H^*(\Omega ^{\infty }D'(2))$ is a conneced
Hopf algebra, by the structure theorem of Hopf algebras over $\Z/2$ (\cite[Theorem 6.1]{Bo53} or \cite[Theorem 7.11]{MM}), this implies
 that $H^*(\Omega ^{\infty }D'(2))$ itself is a polynomial algebra.  Now the Corollary follows as the other factor $H^*(Q_0BO(3))$ is
 polynomial again by \cite[Theorem 3.11]{W}.
\end{proof}



\subsection{Relations among $\mu$-classes}
As an application of the above, we now prove Theorem \ref{relfromdn}.
\begin{defn}
 Define the weight on elements of $H_*(QX)$ by $w(Q^Ix)=2^{l(I)}$ for $x\in Im(H_*(X)\rightarrow H_*(QX))$, and extend it
by $w(xy)=w(x)+w(y)$.  Denote by $W_i(H_*(QX))$ the set of weight $i$ elements.  The module of indecomposables $QH_*(QX)$
inherits
the weights from  $H_*(QX)$, and we have $QH_*(QX)\cong \oplus _iW_{2^i}QH_*(QX)$, one can define similary the weights on elements of $QH_*(\Omega ^{\infty}M(n))$ via the inclusion $QH_*(\Omega ^{\infty}M(n))\subset QH_*(QX)$.
   By the second statement of Lemma \ref{homologysummand}, $QH_*(\Omega ^{\infty}M(n))$ also admits a direct
sum decomposition $QH_*(\Omega ^{\infty}M(n))\cong  \oplus _iW_{2^i}QH_*(\Omega ^{\infty}M(n))$.
In each case, we say that the elements of $W_i(-)$ are homogeneous of weight $i$.
\end{defn}
With this definition, we can state:
\begin{lmm}
 \label{dnhomogeneous}
The map $QH_*(\Omega ^{\infty}d _k):QH_*(\Omega ^{\infty}M(k+1))\rightarrow QH_*(\Omega ^{\infty}M(k))$ sends
homogeneous elements to homogeneous elements, and multiplies the weight by $2$.
\end{lmm}
\begin{proof}
 This is immediate from \cite[Proposition 3.6]{KP1}.
\end{proof}
Define a {\it decreasing} filtration $F_i$ on $QH_*(\Omega ^{\infty }M(n))$ by
$F_iQH_*(\Omega ^{\infty }M(n))=\oplus _{j\geq i} W_j(QH_*(\Omega ^{\infty }M(n)))$.  Then it satisfies the following conditions.
\begin{enumerate}
 \item
$F_1(H_*(\Omega ^{\infty }M(n))=H_*(\Omega ^{\infty }M(n))$
\item if $x\in F_i$ then $Q^s(x)\in F_{2i}$
\end{enumerate}
Thus by naturality of the Dyer-Lashof operations, we see that any map of infinite loop spaces between $\Omega ^{\infty} M(n)$'s
induce filtration preserving map in homology, even though most of the time they don't send homogeneous elements to homogeneous
elements.  We also note that te homology suspension $\sigma ^{\infty }_*$ maps isomorphically $F_1/F_2$ to $H_*(M(n))$.  We will
show the following.
\begin{lmm}
 The map $H_*(\Omega ^{\infty }d _{n-1})$ induces an injection
$$H_*(M(n))\cong F_1/F_2(QH_*(\Omega ^{\infty}M(n)))\rightarrow  F_2/F_4(QH_*(\Omega ^{\infty}M(n-1))).$$
\end{lmm}\begin{proof} By Lemma \ref{dnhomogeneous} for $k=n-1$ we see that $QH_*(\Omega ^{\infty}d_{n-1})$ induces a map from
$F_1/F_2$ to $F_2/F_4$.   By applying  Lemma \ref{dnhomogeneous} to the case $k=n$, we
see that the image of $QH_*(\Omega ^{\infty}d_{n})$ is included in $F_2$. Alternatively,
one can see this by noting that the map $H_*(d_n)=0$.  Thus by the exactness of Proposition
\ref{dnhomologyexact}, $F_1/F_2(QH_*(\Omega ^{\infty }M(n))$ injects to $F_2(QH_*(\Omega ^{\infty }M(n-1))$.  But as we have
$F_1/F_2\cong W_1$ and by Lemma \ref{dnhomogeneous} $W_1$ maps to $W_2$,
we get the desired result.
\end{proof}

Now we are ready to prove Theorem \ref{relfromdn}.  The inclusion $H^*(M(n))\subset H^*(BO(n))$ is given by $H^*(f_n)$, and this is determined uniquely by its compatibility with $H^*(\alpha _n)$, which in turn is determined uniquely by the fact that $H^{-n}(MTO(n))$ contains only one non-trivial element, and the fact that $H^*(D(n))$ is generated by the bottom class as a module over the Steenrod algebra. The cofibration sequence  (\ref{cofibgtmw})
 imply that such a class vanishes if its preimage in $ H^*(QBO(n))$ belongs to the image of $H^*(\Omega _0^{\infty}MTO(n-1))$.
Now, Theorem
\ref{compatibility1} implies that we have a commutative diagram
$$\begin{diagram}
   \node{BO(n)}\arrow{e,t}{f_n}\arrow{s}\node{M(n)}\arrow{s}\arrow{se,t}{d_n}\\
\node{MTO(n-1)}\arrow{e} \node{\Sigma ^{1-n}D(n-1)}\arrow{e}\node{M(n-1)}
  \end{diagram}
$$
Thus we get
$$\begin{diagram}
   \node{H_*(Q_0BO(n)_+)}\arrow{e}\arrow{s}\node{H_*(\Omega ^{\infty }M(n))}\arrow{s}\arrow{se}\\
\node{H_*(\Omega ^{\infty }_0MTO(n-1))}\arrow{e} \node{H_*(\Omega ^{\infty }(\Sigma ^{1-n}D(n-1))}\arrow{e}
\node{H_*(\Omega ^{\infty }M(n-1))}
  \end{diagram}
$$

Dualizing Lemma above, we see that the space of functionals on $QH_*(\Omega ^{\infty }M(n-1))$ vanishing on $F_4$
surjects to the space of functionals on $QH_*(\Omega ^{\infty }M(n))$ vanishing on $F_2$, which is precisely the
image of  $\sigma ^{\infty *}$.  Thus by the commutativity of the diagram above, we see that the image of the composition
$PH^*(M(n))\stackrel{\sigma ^{\infty *}}{\rightarrow} PH^*(\Omega ^{\infty}M(n))\rightarrow PH^*(Q_0BO(n))$ is contained in
the image of $PH^*(\Omega ^{\infty }_0MTO(n-1))$.  This concludes the proof of i).  The other statements follow from
\cite[Theorem 1.8]{KZ1} and Corollary \ref{splitbsdn}.

\begin{rem}
 The formula in \cite[Corollary 4.4']{Raerr} involves a map that sends homogeneous elements
of weight 1 to a sum of homogeneous elements of weight 2 and 4.  The proof of Proposition 4.5' {\it loc. cit.} shows that the terms of
weight 4 can be ignored.  In the above argument, we show that actually one can use another map which are homogeneous.
\end{rem}


To conclude, we give some explicit examples of those relations.  First of all, we have \cite[Corollary 3.11]{MAn}
\begin{prop}
 The image of $H^*(M(n))$ in $H^*(BO(n))$ is the a free-module over $H^*(BT(n))^{Gl_n(\Z /2)}$ generated by
a basis of $A(n-2)Sq^{2^{n-1},\ldots ,2,1}(x_1^{-1}\cdots x_n^{-1})$ where $A(k)$ is the subalgebra of the Steenrod algebra generated by
$Sq^1, Sq^2,\ldots ,Sq^{2^{k}}$. Here we identify $H^*(BO(n))$ with its image in $$H^*(BT(n))\subset H^*(BT(n))^{-\gamma _n})\cong (x_1\cdots x_n)^{-1}H^*(BT(n))$$ via $Bi^*$ where $i:T(n)\subset O(n)$.
\end{prop}
When $n=2$, $A(0)$ is just the exterior algebra generated by $Sq^1$, and $Sq^{2,1}(x_1^{-1}x_2^{-1})=x_1+x_2=\sigma _1$,
$Sq^1(\sigma ^1)=x_1^2+x_2^2=\sigma _1^2$, whereas the Dickson invariants $H^*(BT(n))^{Gl_n(\Z /2)}$ is generated by
$w_2=x_1^2+x_1x_2+x_2^2=\sigma _1^2+\sigma _2$, $w_3=x_1x_2(x_1+x_2)=\sigma _1\sigma _2$, we derive
\begin{cor}
 The set $$\{(\sigma _1^2+\sigma _2)^i(\sigma _1\sigma _j)^j\sigma _1^{\epsilon};i\geq 0, j\geq 0,\epsilon \in \{1,2\}\}$$
forms a basis of the image of $H^*(M(2))$ in $H^*(BO(2))$.
\end{cor}

Below is a table of these relations in low dimensions.
\begin{eqnarray*}
 \mu _{1,0} & = &0\\
\mu _{3,0}+ \mu _{1,1} & = & 0\\
\mu _{2,1} & =& 0\\
\mu _{5,0}+ \mu _{3,1} +\mu _{1,2} & = & 0\\
\mu _{3,1} & = & 0\\
\mu _{4,1} +\mu _{2,2} & = & 0
\end{eqnarray*}
Here we have omitted the relations that follow from lower degree relations and the general relation $\mu _{2i,2j}
=\mu _{i,j}^2$.

\end{document}